\documentclass[12pt]{amsart}
\usepackage{amsbsy,amsthm,amsfonts,amsmath,amssymb,eucal,amscd,graphics,enumerate, ifthen, indentfirst, stmaryrd,xspace,verbatim, epic, eepic,epsfig,verbatim}
\usepackage[all]{xypic}
\include{el_biblio.bib}

\newcommand{\iX}[1]{{\mathcal I}_{\mu_r}({\mathcal X})}

    \def\clE{{\mathcal E}}

\def\clX{{\mathcal X}}
\def\clY{{\mathcal Y}}

\def\bix{{\overline{I}_{\mu}(\clX)}}
\def\biy{{\overline{I}_{\mu}(\clY)}}
\def\bix1{{\overline{I}_{\mu_{r_1}}(\clX)}}
\def\biy2{{\overline{I}_{\mu_{r_2}}(\clY)}}

\newtheorem{definition}{Definition}[section]
\newenvironment{defi}{\begin{definition} \rm}{\end{definition}}

\newtheorem{prop}[definition]{Proposition}

\newtheorem{coro}[definition]{Corollary}
\newtheorem{theo}[definition]{Theorem}
\newtheorem{remark}[definition]{Remark}

\newtheorem{remarkdef}[definition]{Remark-Definition}
\newtheorem{con}[definition]{Conjecture}

\newtheorem{remarks}[definition]{Remarks}

\newtheorem{example}[definition]{Example}

\newtheorem{examples}[definition]{Examples}

\newtheorem{nothing}[definition]{$\!\!$}

\newtheorem{definition*}{Definition}[section]
\newenvironment{defi*}{\begin{definition*} \rm}{\end{definition*}}
\newtheorem{prop*}[definition*]{Proposition}
\newtheorem{lemm*}[definition*]{Lemma}
\newtheorem{coro*}[definition*]{Corollary}
\newtheorem{theo*}[definition*]{Theorem}
\newtheorem{remark*}[definition*]{Remark}
\newenvironment{rema*}{\begin{remark*} \rm}{\end{remark*}}
\newtheorem{remarks*}[definition*]{Remarks}
\newenvironment{remas*}{\begin{remarks*} \rm}{\end{remarks*}}
\newtheorem{example*}[definition*]{Example}
\newenvironment{exam*}{\begin{example*} \rm}{\end{example*}}
\newtheorem{examples*}[definition*]{Examples}\begin{large}                                             \end{large}
\newenvironment{exams*}{\begin{examples*} \rm}{\end{examples*}}

\title{ABELIAN FIBRATIONS AND SYZ MIRROR CONJECTURE}
\address{ Universitat, Departament de Matematiques, Edifici C, Facultad de Ciencies, 08193 Bellaterra, Barcelona}
  \address{Dipartimento di Matematica "Guido Castelnuovo"
Sapienza Universit\`a di Roma
P.le Aldo Moro, 5 - 00185 Roma}
\email{martinez@mat.uniroma1.it}

\author[Cristina Mart{\'\i}nez]{Cristina Mart{\'\i}nez}

\begin{document}
\sloppy
\date{\today}


\begin{abstract}
SYZ mirror conjecture predicts that a Calabi Yau manifold $X$
consists of a family of tori which are dual to a family of special
Lagrangian tori on the mirror dual manifold $\hat{X}$. Here we
consider a fibration of polarized abelian varieties and we construct
a dual one. Moreover we prove that they are equivalent at the level of derived categories.
\vspace{0.5cm}

La conjecture de ``sym\'etrie miroir SYZ'' pr\'edit qu'une vari\'et\'e de Calabi-Yau $X$ consiste en une famille de tores qui sont duaux d'une famille de tores lagrangiennes sp\'eciaux dans la vari\'et\'e miroir duale $\hat X$. Nous consid\'erons ici une fibration de vari\'et\'es ab\'eliennes polaris\'ees et nous en construisons la duale. De plus, nous montrons qu'elles sont \'equivalentes au niveau des cat\'egories d\'eriv\'ees.
\vspace{0.5cm}
\flushright{{\it To the memory of Andrey Todorov}}

\end{abstract}
\maketitle

\section{Introduction}
A Calabi-Yau space has two kind of moduli spaces, the moduli space
of inequivalent complex structures and the moduli space of
symplectic structures. Mirror Symmetry should consist in the
identification of the moduli space of complex structures on an
$n-$dimensional Calabi-Yau manifold $X$ with the moduli space of
complexified K\"ahler structures on the mirror manifold
$\widehat{X}$. 

In the case $X$ is an elliptic curve, the modulus of complex
structures can be identified with the upper half-plane $\mathbb{H}$
by
$$\tau \rightarrow \frac{a\tau+b}{c\tau+d}, \ \ \  A=\left(\begin{array}{ll} a &  b \\
c & d \end{array}\right) \in PSL(2,\mathbb{Z}).$$ We call it
$X_{\tau}$, where $\tau$ is the Teichm\"uller parameter. The second
modulus is the k\"ahler class $[w]\in H^{2}(X,\mathbb{C})$
parametrised by $t\in \mathbb{H}$, as $\int_{X}w=2\pi i t$.

Here we will study Calabi-Yau spaces that are fibred over the same
base $B$ by polarized abelian varieties. If $X/B$ is an abelian
fibration with a global polarization, and we call $X^{\vee}/B$ its
dual fibration, our main result is:

\begin{theo} \label{T1}The derived categories of both fibrations $X/B$ and $X^{\vee}/B$ are equivalent.

\end{theo}

\begin{coro} 
There is an equivalence $\phi_{b}:D^{b}(X_{b})\rightarrow
D^{b}(\widehat{X}_{b})$ for every closed point $b\in B$.
\end{coro}

\section{SYZ mirror conjecture}
There are two main mathematical conjectures in Mirror Symmetry,
Kontsevich homological mirror symmetry conjecture and the conjecture
of Strominger, Yau and Zaslow, which predicts the structure of a CY
manifold and how to get the mirror of a given CY manifold. We first
recall what is a Calabi-Yau manifold.

\begin{defi} A Calabi-Yau manifold $X$ of dimension $n$ is a smooth compact connected $n-$fold with vanishing first Betti number and trivial canonical class $$\Lambda^{n}\Omega_{X}\cong K_{X}\cong \mathcal{O}_{X}.$$
\end{defi}

Let $\pi: X\rightarrow S$ be a proper map, when all the fibers are equidimensional, we say that it is a fibration. We don't impose further assumptions on the base.
It is of interest from the point of view of Mirror Symmetry, the case in which the total fibration is Calabi-Yau. In this case, Strominger, Yau and Zaslow have conjectured how would it be the structure of the mirror fibration. Mirror dual Calabi-Yau manifolds should be fibred over the same base in such a way that generic fibres are dual tori, and each fibre of any of these two fibrations is a Lagrangian submanifold.

\begin{defi} Let $(X,w)$ be a holomorphic symplectic manifold (not necessarily compact) of dimension $2r$. A Lagrangian fibration is a proper map $h:X\rightarrow B$ onto a manifold $B$ such that the general fibre $F$ of $h$ is Lagrangian, that is, $F$ is connected, of dimension $r$, and the restriction $w|_{F}=0$ vanishes. This implies that the smooth fibres of $h$ are complex tori.
\end{defi}

\subsubsection*{SYZ mirror Conjecture.}
If $X$ and $\widetilde{X}$ are a mirror pair of CY $n-$folds, then there exists fibrations $f:X\rightarrow B$ and $\check{f}: \check{X}\rightarrow B$ whose fibres are special Lagrangian with general fibre an $n-$torus.
Furthermore, these fibrations are dual in the sense that canonically $X_{b}=H^{1}(\check{X}_{b}, \mathbb{R}/\mathbb{Z})$ and $\check{X}_{b}=H^{1}(X_{b}, \mathbb{R}/\mathbb{Z})$ whenever $X_{b}$ and $\check{X}_{b}$ are non singular tori. In particular, each of these fibrations admits a canonical section that is an $n-$cycle having intersection number 1 with the fibre cycle. One of the main problems of SYZ mirror conjecture is the presence of singular fibres on the fibration. We will study the moduli problem in this case.

\subsection{Abelian fibrations.}
Let $\Gamma\cong \mathbb{Z}^{d}$ be a lattice in a complex vector space $U$ of dimension $d$ and $\Gamma^{*}\subset U^{*}$
be the dual lattice. The complex torus $(U/\Gamma, I)$ where $I$ is the complex structure, is an abelian variety $A$ of dimension $d$ over $\mathbb{Z}$ if it is algebraic. Let $\widehat{A}$ be the dual abelian variety, i.e the dual torus $(U^{*}/\Gamma^{*}, -I^{t})$. There is a unique line bundle $P$ on the product $A\times \hat{A}$ such that for any point $\alpha\in \widehat{A}$, the restriction $P_{\alpha}$ on $A\times \{\alpha\}$ represents an element of $Pic^{0}(A)$ corresponding to $\alpha$ and in addition, the restriction $P|_{\{0\}\times \widehat{A}}$ is trivial. Such $P$ is called the Poincar\'e line bundle and gives an equivalence between the derived categories of sheaves on $A$ and $\widehat{A}$

\subsubsection{The moduli problem of the dual fibration}

Let  $p:X\rightarrow B$ be a fibration by  abelian varieties with a relative polarization or ample line bundle defined over the total fibration such that the restriction to each fibre is the polarization class on the corresponding fibre. The existence of a relative polarization for the fibration does not necessarily  imply the existence of a section.

\noindent In the fibration $X/B$ singular fibres can appear. In this case, except for some particular cases, we don't know how does the dual abelian variety look like. The idea is to replace the abelian variety by one that is derived equivalent to it.
We consider the moduli problem of the dual fibration, that is, the dual fibration as the stack representing 
the Picard functor, that is, the moduli functor of semistable
sheaves on the fibres that contains line bundles of degree 0 on
smooth fibres. The corresponding coarse moduli space is not a fine
moduli space due to the presence of singular fibres. Let us call
$X^{\vee}$ the dual fibration when it exists and satisfying the
property  that over the smooth locus the fibres correspond to the
dual abelian varieties of the original fibration.

\begin{con}
Two Calabi-Yau threefolds $C_{1}, C_{2}$ that are
fibred over the same base $\mathbb{P}^{1}$ in such a way that the
fibres are abelian surfaces and derived equivalent are derived
equivalent themselves. This prediction is according with SYZ mirror
symmetry conjecture.
\end{con}

First we fix our attention on the smooth locus. Let
$\Sigma(p) \hookrightarrow B$ be the discriminant locus of
$p$, that is, the closed subvariety in the parameter space $B$
corresponding to the singular fibres.
\begin{equation}\label{diagram}\begin{array}{cccc}

 X^{\vee}& \supset X^{\vee}- p^{-1}(\Sigma(p)) & \hookrightarrow & \mathbb{P}^{N} \\
 \downarrow &   & & \\
B &  \supset  B-\Sigma(p) & & \\
\end{array}\end{equation}

Then, we can take  the Zariski closure of
$X^{\vee}-p^{-1}(\Sigma(p))$ in $\mathbb{P}^{N}$. For each $b
\in B-\Sigma(p)$, the corresponding derived equivalence of the
fibres $X_{b}$ and $X^{\vee}_{b}$ is given by the Poincare bundle
$\mathcal{P}_{b}$ over the product $X_{b}\times X^{\vee}_{b}$.
Moreover, its first Chern class $c_{1}(\mathcal{P}_{b})$ lives in
$H^{1,1}(X_{b}\times X^{\vee}_{b},\mathbb{Z})\cap
H^{2}(X_{b}\times X^{\vee}_{b},\mathbb{Z})$. The monodromy group is defined by the action
of the fundamental group of the complement of the discriminant locus $\pi_{1}(B-\Sigma(p))$ on the
cohomology $H^{*}(X_{b}\times X^{\vee}_{b},\mathbb{Z})$ of a fixed non-singular fiber,  and since the class of the polarization is invariant by the monodromy, by Deligne
theorem we can extend the class of the Poincare bundle
 to the non singular fibres, it is the relative Poincare
sheaf of the fibred product of the two families over the base $B$
and we will call it $\mathcal{E}$. In particular, there is a relative polarization and thus we can assume that the fibration is a projective morphism. In the case the fibres are of dimension one, this equivalent to the existence of a multisection. So we will assume the existence of a multisection which means we have a smooth morphism and after an \`etale base change, we get a family of abelian polarized varieties admitting a global section which outside the discriminant locus coincides with the given one.

\begin{prop}\label{prop1}

The Fourier-Mukai transform $$\phi_{\mathcal{E}}: D^{b}(X|_{B-\Sigma\,(p)})\rightarrow D^{b}(X^{\vee}|_{B-\Sigma\,(p)})$$
with kernel $\mathcal{E}$,  is an equivalence of derived categories over the smooth locus. 

\end{prop}

\begin{proof} Let $\widetilde{B}\subset B$ be the open subset supporting the smooth fibres of $p$, and let $X^{sm}:=X|_{B-\Sigma\,(p)}$ be the fibration restricted to the smooth locus $\widetilde{B}:=B-\Sigma\,(p)$. The family has
the structure of a scheme, these are the abelian schemes that can be seen as schemes in groups and have been already studied by Mukai. Due to Deligne (\cite{Del}), there exists a relative polarization meeting transversely any of the irreducible components of any fibre, so there is a global section $\sigma: \widetilde{B}\rightarrow X^{sm}$ obtained by associating to a point $t\in B$, the point $[\mathcal{O}_{X_{t}}]$ corresponding to the semistable sheaf $\mathcal{O}_{X_{t}}$ on the fibre $X_{t}$. 
The Picard functor is representable by the dual fibration $\widehat{X}^{sm}:= X^{\vee}|_{B-\Sigma\,(p)}$.
The dual fibration $\widehat{X}^{sm}$ is defined in such a way that the fibres correspond to the dual abelian varieties of the original fibration, that is, if $\mathcal{E}$ is the relative Poincar\'e sheaf, then $\forall b\in B$, $\mathcal{P}_{b}$ is the Poincar\'e bundle over $X_{b}\times \hat{X}_{b}$. 
There is a natural polarization by considering the product $\pi_{1}^{*}\mathcal{O}_{X^{sm}}(\Theta)\otimes \pi_{2}^{*}\mathcal{O}_{\widehat{X}^{sm}}(\Theta)$, where $\Theta:=\sigma(\widetilde{B})$ and $\pi_{1}$, $\pi_{2}$ are the projection maps of $X^{sm}\times_{B} \widehat{X}^{sm}$ over the first and second components.
The fibres $X_{b}$ and $\hat{X}_{b}$ are derived equivalent, and the equivalence is given by the Poincar\'e bundle over the product $X_{b}\times \hat{X}_{b}$.

Then the Fourier-Mukai transform $\phi_{\mathcal{E}}: D^{b}(X^{sm})\rightarrow D^{b}(\widehat{X}^{sm})$ with kernel $\mathcal{E}$,
\noindent defines an equivalence of the corresponding derived categories. Thus the abelian schemes $X^{sm}$ and $\widehat{X}^{sm}$ are derived equivalent and the equivalence is given by the FMT with kernel the relative Poincar\'e sheaf, $$\phi_{\mathcal{E}}(L)=\mathbf{ R}\pi_{2*}(\pi_{1}^{*}L\otimes \mathcal{E}),$$

\end{proof}

\begin{defi} The dual fibration $(X^{\vee}/B)$ is defined as the moduli stack representing the extended Poincar\'e sheaf $\mathcal{E}$.
\end{defi}

 If $\mathcal{J}$ is the relative moduli functor, and
$Pic^{0}(X/B)$, the relative Jacobian, that is, the variety
$X^{\vee}$ representing the relative moduli functor, the
relative Poincar\'e sheaf $\mathcal{E}$ is the family representing
an element of $\mathcal{J}(Pic^{0}(X/B))$ such that for each
variety $S$ and each $\mathcal{F}\in \mathcal{J}(S)$ there exists
a unique morphism $f:S\rightarrow X^{\vee}$ satisfying that
$\mathcal{F}\cong f^{*}\mathcal{E}$. Therefore $\mathcal{E}$
induces a natural transformation $\Phi: \mathcal{J}\rightarrow
Hom(-,X/B)$ 
giving a stack structure $((X/B),\Phi)$.

It is universal in the sense that for every other variety $N$ and every
natural transformation $$\chi: Hom(-,N) \rightarrow Hom(-,X^{\vee}), $$
the following diagram commutes:

$$\xymatrix{\mathcal{J} \ar[rr]^{\Phi} \ar[dr] & & Hom(-,N)\ar[dl] \\   &   Hom(-,X^{\vee})  } $$

\begin{theo} \label{T1}The extended relative Poincar\'e  sheaf $\clE$ to the total fibration induces a derived equivalence between $X/B$ and $X^{\vee}/B$.
\end{theo}

\begin{proof} 
Since we are assuming there is a multisection $m: B\rightarrow X$, the fibration is a smooth morphism and we can consider an \` etale finite covering $\tau: B'\rightarrow B$ of the base $B$, 
(that is, locally around a point $y\in B'$ and $x=\tau(y)$, $\tau$ is simply the function $\{z\in \mathbb{C}: |z|<1\}\rightarrow \{z\in \mathbb{C} | \, |z|<1\}$ given by $z\rightarrow z^{k}$, where $k$ is the multiplicity of $\tau$ at $y$). Moreover, we can assume that $\tau: B' \rightarrow B$ is a Galois covering with finite Galois group $G$, just we observe that any normal extension of fields admits a Galois extension. If $\mathcal{B},  \mathcal{B}'$ are the fields of meromorphic functions on $B$ and $B'$ respectively, $\tau^{*}: \mathcal{B}\rightarrow \mathcal{B}'$ is a Galois field extension of degree $k$, with Galois group $G$ (acting by pull-back on $\mathcal{B}'$).
$$\xymatrix{ X'=X\times_B B'  \ar[r]\ar[d]  &   X\ar[d]\ar[d]^{\pi}  \\   B' \ar@/^/[u]^{s}\ar[r] &   B }$$

Now the fibration we get $X'\rightarrow B'$ admits a section $s$ that is the pull-back $(\tau\circ m)^{*}$ of the multisection,
and therefore there is a relative Poincar\'e sheaf $\mathcal{E}$ as in the proof of
Proposition \ref{prop1}, that restricted on smooth fibres
$X^{'}_{b}\times \hat{X}'_{b}$, where $\hat{X}'_{b}$ is the
corresponding dual abelian variety, is just the Poincare bundle. We
are taking as dual fibration of $X'/B'$, the relative moduli space.
If $\mathcal{J}$ is the relative moduli functor, $\mathcal{J} : VAR
\rightarrow SETS $, of semistable sheaves of the fibers containing
line bundles of degree 0 on smooth fibres, over the smooth locus,
$\mathcal{J}$ is represented by the relative Jacobian
$Pic^{0}(X'/B')$, which is the dual fibration
$\widehat{X}'^{sm}/B'$.

Due to the presence of singular fibres, the corresponding coarse
moduli space is not a fine moduli space, but the stack maybe a FM
partner, using the relative Poincar\'e sheaf $\mathcal{E}$
as kernel of the FM transform, that is, the relative
Poincar\'e sheaf $\mathcal{E}$ is the family representing an
element of $\mathcal{J}(Pic^{0}(X'/B'))$ such that for each variety
$S$ and each object $\mathcal{F}\in \mathcal{J}(S)$ there exists a
unique morphism $f:S\rightarrow \widehat{X}'$ satisfying that
$\mathcal{F}\cong f^{*}\mathcal{E}$. 

Thus there is an equivalence of categories
$$D^{b}(X'/B')\cong D^{b}(\widehat{X}'/B'),$$
defined by the extended Poincar\'e sheaf $\clE$
 where
$\widehat{\rho}: \widehat{X}'\rightarrow B'$ is the dual abelian
fibration. Now the Galois group $G$ acts on bundles on the fibres, and since they are invariant under the Galois action, there is an equivalence between the respective invariant
subcategories $(D^{b}(X'/B'))^{G}\cong (D^{b}(\widehat{X}'/B'))^{G}$
by the action of the Galois group, therefore by the fundamental
theorem of Galois theory, there is an equivalence between the
original categories $D^{b}(X/B)\cong D^{b}(\widehat{X}/B)$.

\end{proof}

\begin{coro} There is an equivalence $\phi_{b}:D^{b}(X_{b})\rightarrow D^{b}(\widehat{X}_{b})$ for every closed point $b\in B$.
\end{coro}
\begin{proof}
By Theorem \ref{T1}, the integral functor $\phi^{X\rightarrow \widehat{X}}_{\mathcal{E}}: D^{b}(X)\rightarrow D^{b}(\widehat{X})$ is an equivalence of derived categories, where
$\widehat{\rho}: \widehat{X}\rightarrow B$ is the dual abelian
fibration. It follows from Prop. 2.15 of  \cite{HLS} that there is fibrewise equivalence
$\phi_{b}:D^{b}(X_{b})\rightarrow D^{b}(\widehat{X}_{b})$.
\end{proof}

\subsubsection*{Acknowledgments} I would like to thank Yuri Manin who first suggested me the problem of studying the derived categories of Calabi-Yau manifolds and Andrey Todorov for very helpful conversations and comments.
 My work has been partially supported by the project MTM2009-10359

\end{document}